\documentclass[12pt]{amsart}

\usepackage[utf8]{inputenc}     
\usepackage[T1]{fontenc}
\usepackage{lmodern}
\usepackage{graphicx}
\usepackage[english]{babel}
\usepackage{subfigure, enumitem}
\usepackage{amsmath, amsfonts, amssymb, amsthm,dsfont} 
\usepackage{hyperref} 
\usepackage[shortalphabetic]{amsrefs} 
\usepackage{stmaryrd, fancyhdr}
\usepackage[foot]{amsaddr}
\hypersetup{colorlinks = true, urlcolor = blue, bookmarksopen = true}

\def \equi#1{\mathrel{\mathop{\kern 0pt\sim}\limits_{#1}}}

\newcommand*{\legendre}[1]{\left(\frac{#1}{p}\right)}
\newcommand{\Matrix}[1]{\begin{pmatrix} #1 \end{pmatrix}}
\newcommand*{\jacob}[2]{\left(\frac{#1}{#2}\right)}

\def\im{{\rm Im }}
\def\re{{\rm Re }}
\def\Sh{{\rm Sh }}

\newtheorem{prop}{Proposition}
\newtheorem{cor}{Corollary}
\newtheorem{theo}{Theorem}
\newtheorem{lemme}{Lemma}

\theoremstyle{definition}

\newtheorem{rmq}{Remark}

\newcommand{\setC}{\mathbb{C}}

\newcommand{\setQ}{\mathbb{Q}}
\newcommand{\setR}{\mathbb{R}}

\newcommand{\setU}{\mathbb{U}}

\newcommand{\setZ}{\mathbb{Z}}

\newcommand{\A}{\mathcal{A}}
\newcommand{\B}{\mathcal{B}}

\newcommand{\calH}{\mathcal{H}}

\newcommand{\T}{\mathcal{T}}

\numberwithin{equation}{section}

\title{Sign of Fourier coefficients of half-integral weight modular forms in arithmetic progressions}

\author{\textsc{Corentin DARREYE}}
\address{Univ. Bordeaux, CNRS, Bordeaux INP, IMB, UMR 5251,  F-33400, Talence, France}
\email{corentin.darreye@u-bordeaux.fr}

\begin{document}
\renewcommand{\proofname}{Proof}
\renewcommand{\epsilon}{\varepsilon}
\renewcommand{\le}{\leqslant}
\renewcommand{\ge}{\geqslant}

\maketitle

\begin{abstract}

Let $f$ be a half-integral weight cusp form of level $4N$ for odd and squarefree $N$ and let $a(n)$ denote its 
$n^{\rm th}$ normalized
Fourier coefficient. Assuming that all the coefficients $a(n)$ are real, we study the sign of $a(n)$ when $n$ runs through an
 arithmetic progression. As a consequence, we establish a lower bound for the number of integers 
$n\le x$ such that $a(n)>n^{-\alpha}$ where $x$ and $\alpha$ are positive and $f$ is not necessarily a Hecke eigenform. 
\end{abstract}

\tableofcontents

\section{Introduction}

Let $f$ be an element of $S_{\ell+1/2}(4N)$, the space of cusp forms of weight $\ell+1/2$, of level $4N$ and of trivial character 
modulo $4N$. Write its Fourier expansion as 
$$f(z)=\sum_{n\ge1}a(n)n^{\frac{\ell-1/2}{2}}e(nz)$$
for $\im \:z>0$. Given positive real numbers $\alpha$ and $x$ and a class $a$ modulo a prime number $p$, we are interested 
in giving a lower bound on the
 number of integers $n\le x$ such that $n=a\:[p]$ and $a(n)>n^{-\alpha}$ (or $a(n)<-n^{-\alpha}$ respectively). 
 As far as we know, this specific problem 
 for half-integral weight modular forms has not been studied before. Of course, when the weight is an integer, such a question 
 can be partially answered using Sato-Tate equidistribution for Hecke eigenvalues (see \cite[Theorem B]{BLGHT}). 

\vskip 0.5 cm
\subsection{The sign of Fourier coefficients}

In the recent years,  the sign of coefficients of half-integral weight modular forms has drawn considerable attention.
 As a matter of fact, this subject comes from a question asked by Kohnen. Define $f$ as previously and assume that 
 it is a complete
 Hecke eigenform. If $t$ is a positive squarefree integer then, by Waldspurger's  formula, one knows that 
the value of $a(t)^2$ is essentially proportional to the central value $L(1/2,\Sh f\times\chi_t)$ where $\Sh f$ is the Shimura lift
of $f$ and $\chi_t$ is an explicit  Dirichlet character depending on $t$. 

Thus, Kohnen's question is : what squareroot of $L(1/2,\Sh f\times\chi_t)$
 corresponds to $a(t)$? In other words, assume that the coefficients $a(n)$ are real, then what is the sign of $a(t)$ if it is
 nonzero  and what could we say about the other coefficients $a(n)$?
 
 Bruinier and Kohnen \cite{BruiKoh} first showed that under some classical hypothesis,
  the sequence $(a(tn^2))_{n\ge1}$ has infinitely many sign
  changes and equidistribution results were established in \cite{InamWiese} and \cite{Arias} for the sign of this sequence. 
  In \cite{Hulse}, the authors studied the case of the sequence $(a(t))_t$ for all squarefree $t$ and for $f$ of level 4. 
  They showed that there are also infinitely many sign changes in this sequence and  a lower bound for the number of positive
   (respectively negative) $a(t)$ for $t\le x$ is given in \cite{sign}. 
   
   Then, Meher and Murty \cite{MeherMurty} turned their attention to the whole
  sequence $(a(n))_{n\ge1}$ when $f$ is a complete eigenform in Kohnen's plus space of level 4. 
  Detecting the sign changes, they proved in particular that
  \begin{equation}
\label{lowboundplus}
 \big|\big\{n\le x\;\vline\; a(n)\lessgtr 0\big\}\big|\gg_{f,\epsilon} x^{27/70-\epsilon}
\end{equation}
for any $\epsilon>0$. In \cite{sign2}, the authors sharpened the exponent $27/70-\epsilon$ to $1/2$
and generalized it when $n$ runs through an  arithmetic progression of fixed modulus.

Very recently, Lester and Radziwi\l\l\;\cite{LesRad} studied this problem and they showed that under the previous assumptions, 
 there are, for any $\epsilon>0$ and for $x$ large enough, at least $x^{1-\epsilon}$ sign changes in 
 $(a(n))_n$ where $1\le n\le x$ 
 and $n$ is a fundamental 
 discriminant of the form $n=4t$ with $t$ even and squarefree. This improves drastically the bound in \eqref{lowboundplus} but
 actually they did better on this matter. Using a result of Ono and Skinner \cite{OnoSkin}, they gave a rapid and elegant proof 
 of the fact that if $f$ is suitably normalized, then we have
$$ \big|\big\{n\le x\:\vline\: a(n)\lessgtr 0\big\}\big|\gg_{f} \frac{x}{\log x}$$ 
 and the proof can be easily adapted when $f\in S_{\ell+1/2}(4N)$ for any integer $N$ assuming that the hypotheses of 
 \cite[Fundamental Lemma]{OnoSkin} hold and that $f$ is a complete eigenform.

\vskip 0.5 cm
\subsection{Principal results}
Fix $f$ as before. Let  $\alpha$ be a positive real number and let $a$ be a class modulo an integer $q$. 
We consider
\begin{align}
\label{defT+}
\T_{a,q}^+(x ;\alpha)&=\big|\big\{ n\le x \;\vline\;  n=a\:[q] \text{ \:and\: }a(n)> n^{-\alpha} \big\}\big|,\\
\T_{a,q}^-(x;\alpha)&=\big|\big\{   n\le x \;\vline\; n=a\:[q] \text{\: and\: }a(n)< -n^{-\alpha} \big\}\big|
\label{defT-}
\end{align}
 for any $x>0$. We also put $\T^\pm(x;\alpha)=\T_{0,1}^\pm(x ;\alpha)$.
 
 Using recent results on sums of Fourier coefficients of half-integral weight modular forms in  arithmetic progressions, 
 we prove a lower bound for  
 $\T_{a,p}^\pm(x ;\alpha)$ for some fixed $\alpha$ and for a positive proportion of $a\:[p]$. Here, $p$ is a prime such that $p$ 
 and $x$ are both going to infinity in a certain range. 
 
 We distinguish two cases. The first one is when $f$ is not necessarily a Hecke eigenform.
  In that case, we prove that there exists a positive proportion of $a\:[p]$\footnote{Here and in the rest of the paper, a positive 
 proportion means a number of $a\:[p]$ which is $\gg p$} such that at least one coefficient $a(n)$ with $n=a\:[p]$ and 
 $n\le x$ satisfies $a(n)> n^{-\alpha}$ for some positive $\alpha$.
 
 The second case is when $f$ is a complete Hecke eigenform. Then, we establish a lower bound on the number of $a\:[p]$
 such that both $\T_{a,p}^+(x ;\alpha)$ and $\T_{a,p}^-(x ;\alpha)$ are bigger than $\frac{x^{1-\epsilon}}{p^{7/4}}$ for 
 positive $ \epsilon$.
 
 We emphasize the fact that there are two novelties in this work. 
 First, we are not just looking at the signs of the coefficients, we also provide lower bounds for $|a(n)|$. 
 Moreover, we include the case where $f$ is not necessarily a Hecke eigenform which is significantly 
  different from the previous papers about the signs of coefficients of half-integral weight cusp forms. 
 Indeed, in the works we mentioned above, the assumption that $f$ is an eigenform is crucial since,  in that case,  
 Shimura's correspondence is quite explicit on the coefficients (see \eqref{relationShiExplicite} below)  and one can
  apply Waldspurger's formula.
  \\
 
 Let us now state the main theorem of this paper.

\begin{theo}
\label{theo-lowboundT}
Let $f\in S_{\ell+1/2}(4N)\backslash\{0\}$ where $\ell$ and $N$ are two positive integers with $N$ odd and squarefree. 
If $\ell=1$, we assume that $f$ is
in the orthogonal complement of the subspace spanned by single variable theta-functions. We also assume that the
 Fourier coefficients of $f$ are real.

Then, for any $\epsilon>0$ and any $\alpha\in(3/14,1/4]$, there exists a constant $x_0=x_0(f,\epsilon,\alpha)$ such that for all  
$x_0\le x^{1-2\alpha+\epsilon}\ll p\ll x^{4/7-\epsilon}$ with $p$ a prime number, we have
$$\T_{a,p}^+(x ;\alpha)\ge 1$$
 for a positive proportion of $a\:[p]$ and where $\T_{a,p}^+(x ;\alpha)$  is defined in \eqref{defT+}. The same holds for
 $\T_{a,p}^-(x ;\alpha)$.

If, moreover, we assume that $f$ is a complete Hecke eigenform, then for any $\epsilon>0$, any $\delta>0$ small enough 
and any $\alpha\in(1/8,1/7]$, there exists a constant $x_0=x_0(f,\epsilon, \delta, \alpha)$ such that for all  
$x_0\le x^{1/2+\epsilon}\ll p\ll x^{4\alpha-\epsilon}$ with $p$ a prime number, we have
$$\T_{a,p}^\pm(x ;\alpha)\gg \frac{x^{1-2\delta}}{p^{7/4}}$$
 for a number of $a\:[p]$ which is $\gg_{f,\delta} \frac{p^{3/4}}{x^{\delta/2}}$.

\end{theo}

We will then deduce the following corollary.

\begin{cor}\label{cor-sign}
Let $f\in S_{\ell+1/2}(4N)\backslash\{0\}$ as in Theorem \ref{theo-lowboundT} (but not necessarily a Hecke eigenform)
 with $N$ odd and squarefree. Then,
$$ \T^\pm(x;3/14+ \epsilon)\gg x^{4/7-\epsilon}$$ 
for any $\epsilon>0$ and $x$ large enough.
\end{cor}

\begin{rmq}
The previous corollary implies an omega result on the absolute value of $a(n)$. However, the conclusion reached is weaker than
the one established in recent papers on this subject (see \cite{Das} and \cite{GKS}).
\end{rmq}

The proof of the first assertion of Theorem \ref{theo-lowboundT} is based on estimates on sums of Fourier 
coefficients over arithmetic progressions. This type of sums has drawn particular
interest  over the past decade, especially for integral weight modular forms (see for example \cite{LauZhao} and \cite{FGKM}).
The case of half-integral weights was treated in \cite{Darreye}. We will need the following.

\begin{theo}\label{theo-moment24}
Let $f\in S_{\ell+1/2}(4N)\backslash\{0\}$ and $w$ a smooth real-valued function compactly supported 
in $(0,+\infty)$. Define for any $x>0$, any prime number $p$ and any class $a$ modulo $p$
$$E(x,p,a)=\frac{1}{\sqrt{x/p}}{\sum_{n=a\:[p]}} a(n)w(n/x).$$

Then, for any $\epsilon>0$,
$$\frac{1}{p}{\sum_{a\:[p]}}^\times |E(x,p,a)|^2 \sim c_f \| w\|_2^2$$
as long as $x^{1/2+\epsilon}\ll p\ll x^{1-\epsilon}$. The symbol ${\sum}^\times$ means we restrict the
 summation over invertible classes modulo $p$, $\| w\|_2$ is the $L^2$ norm of $w$ and $c_{f}$ is a positive constant depending
 only on $f$.

Moreover, if we assume that $N$ is odd and squarefree, that $f$ is
in the orthogonal complement of the subspace spanned by single variable theta-functions when $\ell=1$ and that the Fourier
coefficients of $f$ are real, then 
$$\frac{1}{p}{\sum_{a\:[p]}}^\times E(x,p,a)^4 \le 12(c_f \| w\|_2^2)^2 + o(1)$$
as long as $x^{1/2+\epsilon}\ll p\ll x^{4/7-\epsilon}$.
\end{theo}

This is a special case of  \cite[Theorem 3]{Darreye} where the assumption that $f$ is a complete eigenform was relaxed and 
the level of $f$ is greater than 4.

In order to prove the second assertion of Theorem \ref{theo-lowboundT}, we will need the following result
about the fourth moment of the Fourier coefficients. While the second moment can be easily computed using the
classical theory of Rankin-Selberg transform,  the fourth moment is more tricky to estimate. As it is done in \cite{LesRad}, 
we will do it by using Waldspurger's formula \cite{Wald} and a large-sieve type inequality by Heath-Brown for quadratic characters 
\cite{Heath-Brown}.

\begin{prop}\label{prop-4moment}
Let $f\in S_{\ell+1/2}(4N)$ be a complete Hecke eigenform with $N$ odd and squarefree. If $\ell=1$, we assume that $f$ is
in the orthogonal complement of the subspace spanned by single variable theta-functions. Then 
$$\sum_{n\le x}|a(n)|^4\ll_{f,\epsilon} x^{1+\epsilon}$$
for any $\epsilon>0$ and any $x>0$.
\end{prop}

 \begin{rmq}
The assumption that $N$ is odd and squarefree in Theorem  \ref{theo-moment24} comes from the
fact that we need a sufficiently good theory on newforms of half-integral weight. As far as we know, such a theory doesn't exist
 on $S_{\ell+1/2}(4N)$ for arbitrary $N$.
 
 We also need this assumption in Proposition \ref{prop-4moment} to make Waldspurger's formula a bit more explicit.  
\end{rmq}

\vskip 0.5 cm
\subsection{Structure of the paper}
Since we are going to work with smooth sums, we will consider coefficients $a(n)$ with $n/x$ in the compact support
of a smooth function $w$ on $(0,+\infty)$ and  prove a lower bound for
\begin{equation}
\label{defTsmooth}
\T_{a,q}^+(x,\alpha; w)=\big|\big\{n\ge1 \;\vline\;   n=a\:[q] \text{ \:and\: } a(n)w(n/x)> n^{-\alpha}w(n/x) \big\}\big|.
\end{equation}
The case of $\T_{a,q}^-(x,\alpha; w)$ will follow easily by changing $f$ in $-f$. 

Without loss of generality, we may assume that $w$ is supported in $(0,1)$ and takes values in $[0,1]$.\\

We will proceed as follow. After proving Theorem \ref{theo-moment24}, we will combine it with Hölder's inequality to show that
$$\sum_{a\:[p]}|E(x,p,a)|\gg p$$
 which yields $E(x,p,a)\gg 1$ for a positive proportion of $a\:[p]$ since $\sum\limits_{a\:[p]}E(x,p,a)$ is small. 
 Then, the first assertion of Theorem \ref{theo-lowboundT}
 follows from an easy counting argument. We will also prove the second assertion of Theorem \ref{theo-lowboundT} 
 following the same line but we will use Proposition \ref{prop-4moment} instead of the result about the fourth moment in arithmetic
 progression.\\

We will first recap some basic facts about half-integral weight modular forms and prove Proposition \ref{prop-4moment}
in Section \ref{Modular}. Section \ref{Mainestimates}
is dedicated to the proof of Theorem \ref{theo-moment24} while Theorem \ref{theo-lowboundT} will be
proved in Section \ref{ProofofTheo}.

\vskip 0.5 cm
\subsection{Notations}
As usual, we write $a\:[p]$ for a class $a$ modulo $p$ and we also put $e_{p}(a)=e(a/p)$ with $e(x)=e^{2i\pi x}$.

The group $\text{GL}_2(\setR)^+$ (consisting of real matrices of positive determinant) acts on the 
Poincaré half-plane $\calH $  by M\"obius transformation and we write this action as 
$$\gamma z= \frac{az+b}{cz+d}$$
for any $\gamma =\Matrix{a&b\\c&d}\in\text{GL}_2(\setR)^+$ and any $z\in\calH$. 

We also denoted by $I_{2}$ the identity matrix in $\text{GL}_2(\setR)^+$ and by $\Gamma_0(N)$ the usual congruence 
subgroup.

For any odd integer $d$, define $\epsilon_d$ as the normalized Gauss sum \emph{i.e.} 
$ \epsilon_{d}= \left\{\begin{array}{cc}
1 & \text{ if } d=1\:[4],\\
i & \text{ if } d=3\:[4],
\end{array}\right.$
and for any fundamental discriminant $D$, we denote by $\jacob{D}{\cdot}$ its associated quadratic character.  More generally,
 any non-zero integer $n$, with $n=0,1\:[4]$, can be written in a unique way as $n=Dm^{2}$ where $D$ is a fundamental
discriminant and $m\in\setZ$.  Hence, we denote by $\jacob{n}{\cdot}$ the character modulo $|n|$ induced by $\jacob{D}{\cdot}$.
If $n=2,3\:[4]$ then $4n$ can be written in a unique way as $4n=Dm^{2}$ where $D$ is a fundamental
discriminant and $m\in\setZ$.  In this case, we denote by $\jacob{n}{\cdot}$ the character modulo 
$4|n|$ induced by $\jacob{D}{\cdot}$. By convention, we also let $\jacob{0}{\pm1}=1$.

 If $x$ is a square modulo an odd prime $p$, we denoted by $\sqrt x^p$ the only integer $y\in [1,(p-1)/2]$ such that 
 $x=y^{2}\:[p]$.

Finally, the symbol ${\sum}^\flat$  means we restrict the summation to positive squarefree integers $t$ and we put 
$ \delta_p(x)= \left\{\begin{array}{cc}
1 & \text{ if } x=0\:[p],\\
0 & \text{ otherwise.}
\end{array}\right.$

\vskip 0.5 cm
\subsection{Acknowledgements}
The author would like to express his gratitude to Florent Jouve and Guillaume Ricotta for their many helpful comments and useful 
suggestions.

\vskip 1cm
\section{Modular forms of half-integral weight}
\label{Modular}
In this section we first recall the principal properties of half-integral weight modular forms that we will use in this
paper. A good introduction to this theory can be found in \cite{Ono} and a more complete study is done in \cite{Shi} or 
\cite{Knapp}. Then, we will prove Proposition \ref{prop-4moment} and a non trivial bound for Fourier coefficients of such forms.

\vskip 0.5 cm
\subsection{General setting}
Let $G$ be the set of pairs $(\sigma,\phi)$ where $\sigma=\Matrix{a&b\\c&d}\in\text{GL}_2(\setR)^+$ and 
$\phi:\calH\to\setC$ is a holomorphic function such that $\phi(z)^2=\eta(\det \sigma)^{-1/2} (cz+d)$ for all $z\in\calH$ and with
 $\eta$ a complex 
number of norm 1 not depending on $z$. $G$ has a group structure with the inner law defined by
$$(\sigma,\phi)(\sigma',\phi')=(\sigma\sigma',\phi(\sigma'\cdot)\phi').$$

This group is a non-trivial central extension of $\text{GL}_2(\setR)^+$ by $\setU$ the unit circle \emph{i.e.} the sequence
$$1\to \setU \to G \to \text{GL}_2(\setR)^+\to1,$$ where $\eta\in\setU$ is sent to $(I_2,\eta)$, is exact and the center of $G$ is 
the subgroup of pairs $(\alpha I_2,\eta)$ with $\alpha\in\setR^*$ and $\eta\in\setU$.

This sequence splits over $\Gamma_0(4)$ which means this group has a section $s_J:\Gamma_0(4)\to G$ given
explicitly by $s_J(\gamma)=(\gamma,J(\gamma,z))$ with
$$J(\gamma,z)=\epsilon_d^{-1}\jacob{c}{d}\sqrt{cz+d}$$ 
for any $\gamma=\Matrix{a&b\\c&d}\in\Gamma_0(4)$. For any positive integer $N$, we denote by $\Delta_0(4N)$ the image
of $\Gamma_0(4N)$ by $s_J$. If $\chi$ is a Dirichlet character modulo $4N$, then we put $\chi(\xi)=\chi(d)$ for any 
$\xi=\left(\Matrix{a&b\\c&d},\phi\right)\in\Delta_0(4N)$.

Now for any integer $\ell$, any function $f:\calH\to\setC$ and any $\xi=(\sigma,\phi)\in G$, we define the weighted slash operator 
$|_{\ell+1/2}$ by $$f|_{\ell+1/2}\xi(z) = \phi(z)^{-(2\ell+1)}f(\sigma z)$$
which gives a well-defined right action of $G$ on such functions $f$.\\

We say that $f$ is a modular (respectively a cusp) form of level $4N$, of weight $\ell+1/2$ and of character $\chi$, and we note
$f\in M_{\ell+1/2}(4N,\chi)$ (respectively $f\in S_{\ell+1/2}(4N,\chi)$), if 
\begin{enumerate}
\item $f$ is holomorphic on $\calH$,
\item $f|_{\ell+1/2}\xi=\chi(\xi)f$ for all $\xi\in\Delta_0(4N)$,
\item $f$ is holomorphic (respectively cuspidal) at each cusp of $\Gamma_0(4N)$.
\end{enumerate}

The third point means that for all cusp $\mathfrak a$ of the curve $\Gamma_0(4N)\backslash\calH$ and for all element 
$\xi_\mathfrak a = (\sigma_\mathfrak a,\phi_\mathfrak a)\in G$ with $\sigma_\mathfrak a \infty =\mathfrak a$, the function
$f|_{\ell+1/2}\xi_\mathfrak a$ has a Fourier expansion with only non-negative (respectively positive) powers of 
$e(z/r_{\mathfrak a})$ for some positive integer $r_{\mathfrak a}$.\\

The Hecke operators $T_m$ are defined on $M_{\ell+1/2}(4N,\chi)$ as double coset operators for 
 $\Delta_0(4N)\xi_m\Delta_0(4N)$ with $\xi_m=\left(\Matrix{1&0\\0&m},m^{1/4}\right)$. In particular (see \cite{Shi}), they vanish
  when $m$ is not a square and they satisfy the multiplicativity relation $T_{mn}=T_mT_n$ for $(m,n)=1$.
  
  Also, the $T_{p^2}$ for $p\nmid4N$ are normal  operators (they are self-adjoint when $\chi$ is real) 
  on $S_{\ell+1/2}(4N,\chi)$ with
   respect to the Petersson inner product. So there exists a basis of  $S_{\ell+1/2}(4N,\chi)$ composed of common eigenfunctions
   of all the $T_{p^2}$ for $p\nmid4N$ which we call eigenforms. When $N$ is squarefree, with a suitable theory of newforms of
    half-integral weight (see \cite{Manickam90} and \cite{Manickam11}) one can prove that some of these eigenforms are actually
    also eigenfunctions for $T_{p^2}$ with $p\mid4N$. We call them complete eigenforms.  
    
 When $\chi$ is trivial, since there exists a basis of $S_{\ell+1/2}(4N)$ composed of forms with rational coefficients 
    (see \cite{BelabasCohen}) and the Hecke operators are rational on this space, then there exists a non-trivial subspace of
     $S_{\ell+1/2}(4N)$ spanned by eigenforms (or even complete eigenforms if $N$ is squarefree) which have real Fourier
     coefficients. \\

For $f\in S_{\ell+1/2}(4N,\chi)$, we denote by $f_0$ its image under the Fricke involution \emph{i.e.}
$$f_0=f|_{\ell+1/2}W_{4N}$$
where $W_{4N}=\left(\Matrix{0&-1\\4N&0}, (4N)^{1/4}\sqrt{-iz}\right)$. Then 
$f_{0}\in S_{\ell+1/2}\left(4N,\jacob{4N}{\cdot}\overline\chi\right)$.

\vskip 0.5 cm
\subsection{Shimura's correspondence and Waldspurger's Theorem}
Let $f\in S_{\ell+1/2}(4N,\chi)$ be a complete eigenform with eigenvalues $(\lambda(p))_p$ \emph{i.e.}
$$T_{p^2}f=\lambda(p)f$$ for any prime $p$. Define $\lambda(n)$ for any integer $n$ formally by 
$$\sum_{n\ge1}\lambda(n)n^{-s}=\prod_p\big(1-\lambda(p)p^{-s}+\chi^2(p)p^{2\ell-1-2s}\big)^{-1}.$$

Then, Shimura \cite{Shi} and Niwa \cite{Niwa} showed that the function defined by 
$$\Sh f(z)=\sum_{n\ge1}\lambda(n)e(nz)$$ 
for $z\in\calH$ is a complete Hecke eigenform in  $S_{2\ell}(2N,\chi^2)$ whenever $\ell\ge2$. 
For $\ell=1$, this holds if one assumes
that $f$ is in the orthogonal complement of the subspace spanned by single variable theta-functions (we always make this
assumption in the sequel). 

Moreover, for any integer $t$ which is not divisible by a square prime to $4N$, we have
\begin{equation}
\label{relationShiExplicite}
a(tn^2)n^{\ell-1/2}=a(t)\sum_{d\mid n}\mu(d)\jacob{(-1)^\ell t}{d}\chi(d)d^{\ell-1}\lambda(n/d)
\end{equation}
for all integer $n$ and where $a(n)$ is the $n^{\rm th}$ normalized Fourier coefficient of $f$. 

Therefore, by Deligne's bound for Hecke eigenvalues for integral weight modular forms \cite{weil1}, one has
\begin{equation}\label{boundcoefftn2}
|a(tn^2)|\ll_\epsilon |a(t)| n^{\epsilon}
\end{equation}
 for any $\epsilon>0$. 
 
 Waldspurger's formula relates $a(t)$ to the central value of the $L$-function associated to $\Sh f$ twisted by 
 a character. We give a statement for such a formula which can be easily derived from \cite[Théorème 1]{Wald}. For more
 explicit formulas, see also \cite{KohnenZagier}, \cite{Kohnen85} and \cite{Shi2}.
 
 Let $f\in S_{\ell+1/2}(4N,\chi)$ as before and assume that $N$ is odd. For squarefree $t$, let us consider the Dirichlet character
 \begin{equation}
\label{defchit}
\chi_t=\jacob{(-1)^\ell t}{\cdot}\chi
\end{equation}
whose conductor divides $4Nt$.
 
 By Shimura's correspondence and Atkin-Lehner theory \cite{AtkinLehner}, there exists a unique newform 
 $$F(z)=\sum_{n\ge1}b(n)n^{\ell-1/2}e(nz)$$
 in $S_{2\ell}(M,\chi^2)$ for some $M\mid 2N$ such that 
  $b(p)p^{\ell-1/2}=\lambda(p)$ for all prime $p\nmid 2N$. Define the twisted form $F_t$ by
  $$F_t(z)=\sum_{n\ge1}\overline \chi_t(n)b(n)n^{\ell-1/2}e(nz)$$
 which is in $S_{2\ell}(16N^2t^2)$ (see \cite[Proposition 3.1]{AtkinLi}). It is an eigenfunction of the $p^{\rm th}$ Hecke operator for $p\nmid2Nt$  whose
 eigenvalue is $\overline \chi_t(p)\lambda(p)$. Therefore, there exists a unique newform 
 \begin{equation}
 \label{defFtilde}
 \widetilde F_t(z)=\sum_{n\ge1}b_t(n)n^{\ell-1/2}e(nz)
 \end{equation}
  in $S_{2\ell}(M')$ for some $M'\mid 16N^2t^2$ such that 
  $b_t(p)p^{\ell-1/2}=\overline \chi_t(p)\lambda(p)$ for all prime $p\nmid 2Nt$. We define its normalized $L$-function as
  $$L(s,\widetilde F_t)=\sum_{n\ge1}\frac{b_t(n)}{n^s}$$ which converges absolutely for $\re\: s>1$.
 
 \begin{theo}\cite[Théorème 1]{Wald}\label{Walds}
 With notations as above, there exists a bounded function $c_f(t)$ defined on squarefree
  integers and depending only on $f$ such that for all squarefree $t$,
   $$a(t)^2= c_f(t) L(1/2,\widetilde F_{t}).$$
 \end{theo}
 
We will deduce from this theorem the estimate we need for the  fourth moment of the coefficients $a(n)$.
\vskip 0.5 cm
\subsection{The fourth moment} The goal of this subsection is to prove Proposition \ref{prop-4moment}. 
The idea is to exploit  \eqref{boundcoefftn2} and Theorem \ref{Walds} to reduce this problem to finding an estimation of
\begin{equation}
\label{sumquadratictwist}
\sum_\psi L(1/2,g\times \psi)^2
\end{equation}
where $\psi$ runs through quadratic characters of bounded conductors and $g$ is some newform
of integral weight.

Since the form $\widetilde F_{t}$ is not equal to $F_t$ in general (their $L$-functions are equal up to a finite number of Euler 
factors but this number could increase with $t$), we will need assumptions under which $\widetilde F_{t}$ 
is actually the twist of $F$ by a quadratic character. Hence, we first prove the following lemma.

\begin{lemme}\label{lemma-twist}
Let $F(z)=\sum\limits_{n\ge1}\lambda(n)e(nz)\in S_{k}^{\rm new}(N)$ be a complete Hecke eigenform and let $\psi$ be a
 primitive quadratic character modulo $M$. If $N$ is squarefree then the form 
 $$F_\psi(z)=\sum\limits_{n\ge1}\psi(n)\lambda(n)e(nz)$$ is a newform and a complete Hecke eigenform.
\end{lemme}

\begin{proof}
By assumption, either $M$ is squarefree or it can be written as $M=4t$ with $t$ squarefree. Then, $\psi$ decomposes as a
product of primitive characters
$$\psi=\prod_{p\mid M}\psi_p$$
where $\psi_p=\jacob{\cdot}{p}$ for odd $p$ and $\psi_2$ is a primitive character of conductor 4 or 8. 
By \cite[Theorem 6]{AtkinLehner}, for any newform $G$ of level $N'$ with  $p$-adic valuation  $v_p(N')=0$ or $1$, we have
$G_{\psi_p}\in S_{k}^{\rm new}(N'p^{2-v_p(N')})$. It is also suggested in \cite{AtkinLehner} that this holds for quadratic characters
 modulo 4 and 8 but since it is not explicitly written, we prefer to refer to \cite[Theorem 3.1 and Corollary 3.1]{AtkinLi} 
 from which we can deduce that if $M=2^\alpha M'$ with $\alpha\in\{2,3\}$ and $M'$ odd, then 
 $F_{\psi_2}\in S_{k}^{\rm new}(2^{2\alpha-v_2(N)}N)$. Then, it follows easily that $F_\psi$ is a newform and since it is
 an eigenfunction of all but finitely many Hecke operators $T_p$, it must be a complete eigenform.
\end{proof}

We now deduce the following classical estimate for sums of type \eqref{sumquadratictwist}.

\begin{prop}\label{prop-sumquadra}
Let $F(z)=\sum\limits_{n\ge1}a(n)n^{\frac{k-1}{2}}e(nz)\in S_{k}^{\rm new}(N)$ be a complete Hecke eigenform with $N$ 
squarefree. For $x>0$, let $\Psi(x)$ denote the set of primitive quadratic characters of conductor at most $x$. Then
$$\sum_{\psi\in\Psi(x)}|L(1/2,F_\psi)|^2\ll_{F,\epsilon} x^{1+\epsilon}$$
for any $\epsilon>0$ and where $L(s,F_\psi)=\sum\limits_{n\ge1}\psi(n)a(n)n^{-s}$.
\end{prop}

\begin{proof}
By Lemma \ref{lemma-twist}, for all $\psi\in\Psi(x)$, $F_\psi$ is a newform whose $L$-function satisfies a functional equation of 
the form
$$\Lambda(s,F_\psi) := N_\psi^{s/2}(2\pi)^{-s}\Gamma(s+(k-1)/2)L(s,F_\psi)= \epsilon(F_\psi)\Lambda(1-s,F_\psi)$$
for some $\epsilon(F_\psi)\in\{\pm1\}$ and where $N_\psi\le Nx^2$ is the level of $F_\psi$.

Then, using the approximate functional equation (see \cite[Theorem 5.3]{IwKo}) one derives
\begin{equation}
\label{approxfuncteq}
L(1/2,F_\psi)=(1+\epsilon(F_\psi))\sum_{n\ge1}\frac{\psi(n)a(n)}{\sqrt n}V_\frac{1}{2}\left(\frac{n}{\sqrt{N_\psi}}\right)
\end{equation}
where $$V_\frac{1}{2}(y)=\frac{1}{2i\pi}\int_{(\sigma)}\frac{\Gamma(s+k/2)}{s\Gamma(k/2)}(2\pi y)^{-s}ds$$ for any $\sigma>0$.
Let $\eta>0$. Breaking the sum in \eqref{approxfuncteq} according to $n<x^{1+\eta}$ or not and using the fact that 
$V_\frac{1}{2}(y)\ll_{k,A}y^{-A}$ 
for any $A>0$, we have 
$$\sum_{\psi\in\Psi(x)}|L(1/2,F_\psi)|^2\ll_{F,\eta}
\sum_{\psi\in\Psi(x)}\left|\sum_{n<x^{1+\eta}}\frac{\psi(n)a(n)}{\sqrt n}V_\frac{1}{2}\left(\frac{n}{\sqrt{N_\psi}}\right)\right|^2.$$

Now, by \cite[Corollary 2]{Heath-Brown}, the right-hand side of the above inequality is 
$$\ll_\epsilon x^{(2+\eta)\epsilon+1+\eta}\sum_{\tiny \begin{array}{c}n_1,n_2< x^{1+\eta}\\ \sqrt{n_{1}n_{2}}\in\setZ \end{array}}
\frac{|a(n_1)a(n_2)|}{\sqrt{n_1n_2}}$$
and this last sum is bounded by $\sum\limits_{n<x^{1+\eta}}\frac{\sigma_{0}(n^2)}{n^{1-\eta}}$ with $\sigma_{0}(n^2)$ 
the number of divisors of  $n^2$. Since $\eta$ can be arbitrary small, the conclusion follows.

\end{proof}

We can now prove Proposition \ref{prop-4moment}.

\begin{proof}[Proof of Proposition \ref{prop-4moment}]
Write $f(z)=\sum\limits_{n\ge1}a(n)n^{\frac{\ell-1/2}{2}}e(nz)\in S_{\ell+1/2}(4N)$ as usual. 

Since $N$ is odd, by Theorem \ref{Walds}, $$|a(t)|^2\ll_f|L(1/2,\widetilde F_t)|$$ 
for any squarefree $t$ and where $\widetilde F_t$ is defined by \eqref{defFtilde}. 
Let $D_t$ be the fundamental discriminant such that $\jacob{D_t}{\cdot}$ induces $\chi_t$ (defined by \eqref{defchit} 
for $\chi$ principal). The discussion before Theorem \ref{Walds} and Lemma \ref{lemma-twist} show
that $\widetilde F_t$ is actually the twist by $\jacob{D_t}{\cdot}$ of a complete eigen-newform $F$ which depends only on $f$.

Thus, by \eqref{boundcoefftn2}, we have for any $\epsilon>0$,
\begin{align*}
\sum_{n\le x}|a(n)|^4&\ll_{f,\epsilon}{\sum_{t\le x}}^\flat |L(1/2, F_{\jacob{D_t}{\cdot}})|^2
\sum_{m\le\sqrt{\frac{x}{t}}}m^{2\epsilon}\\
&\ll_{f,\epsilon}x^{1/2+\epsilon}{\sum_{t\le x}}^\flat|L(1/2, F_{\jacob{D_t}{\cdot}})|^2t^{-1/2-\epsilon}
\end{align*}
and since $|D_t|\le4t$, a summation by parts and Proposition \ref{prop-sumquadra} give the result.
\end{proof}

\vskip 0.5 cm
\subsection{Bounds for Fourier coefficients}
Let $f\in S_{\ell+1/2}(4N,\chi)$ and put
$$f(z)=\sum_{n\ge1}a(n)n^{\frac{\ell-1/2}{2}}e(nz)$$ for $z\in\calH$.
 In the proof of Theorem \ref{theo-moment24}, we will use a bound for the coefficients $a(n)$ which must hold for arbitrary $f$.
  
 If $\ell\ge2$ and $t$ is squarefree then, by \cite[Theorem 1]{Iwcoef}, one has
$$|a(t)|\ll_{f,\epsilon} t^{3/14+\epsilon}$$
for any $\epsilon>0$. Actually this still holds if $t$ is divisible by $p^2$
for $p\mid4N$. Precisely, we have the following proposition.

\begin{prop}\label{prop-boundcoeffsqfr}
Let $f\in S_{\ell+1/2}(4N,\chi)$ where $\ell$ and $N$ are two positive integers. If $\ell=1$, we assume that $f$ is
in the orthogonal complement of the subspace spanned by single variable theta-functions. For all integer $n=tm^2$ 
with squarefree $t$ and $m\mid4N$, we have 
$$|a(n)|\ll_{f,\epsilon}n^{3/14+\epsilon}$$ for any $\epsilon>0$.
\end{prop}

\begin{proof}
This is a straightforward consequence of \cite[Theorem 1]{Waibel}.
\end{proof}

From this we can deduce the following more general bound.

\begin{prop}\label{prop-boundcoeffgene}
Assume the hypotheses of the previous proposition hold. If, moreover, $N$ is odd and squarefree and $\chi$ is real then
 for all squarefree $t$ and all positive integer $n$,
$$|a(tn^2)|\ll_{f,\epsilon}t^{3/14}(tn)^\epsilon$$ for any $\epsilon>0$.

\end{prop}

\begin{proof}
By \cite[Theorems 7]{Manickam90}, $f$ can be written as a finite sum
$$f=\sum_{i}U(r_i^2)f_{i}$$
where $r_i\mid2N$, $U(r_i^2) : \sum\limits_{n\ge1}c(n)e(nz)\mapsto\sum\limits_{n\ge1}c(r_{i}^2n)e(nz)$ and $f_i$ is 
either a complete eigenform of $S_{\ell+1/2}(M,\chi)$ for some $M$ dividing $4N$ or 
a complete eigenform of Kohnen's plus space
$S_{\ell+1/2}^+(M,\chi)$ for some $M$ dividing $4N$. Hence,
$$a(tn^2)=\sum_{i}a_i(t(r_in)^2)$$ with $a_i(m)$ the $m^{\rm th}$ normalized coefficient of $f_i$. 
If $f_i$ is a classical eigenform then, by \eqref{boundcoefftn2} and Proposition \ref{prop-boundcoeffsqfr}, we have
$|a_i(t(r_in)^2)|\ll_{i,\epsilon} t^{3/14}(tr_in)^\epsilon$.

If $f_i$ is in the plus space, then relation \eqref{boundcoefftn2} still holds but with $t=|D|$ where $D$ is a fundamental
 discriminant (see \cite{Kohnen82}). In that case we have $|a_i(t(r_in)^2)|\ll_{i,\epsilon} |a_i(t)|(r_in)^\epsilon$ or
  $|a_i(t(r_in)^2)|\ll_{i,\epsilon} |a_i(4t)|(r_in/2)^\epsilon$. In both cases, we can apply Proposition \ref{prop-boundcoeffsqfr} 
  and get $|a_i(t(r_in)^2)|\ll_{i,\epsilon}t^{3/14}(tr_in)^\epsilon$ which is enough to conclude.
\end{proof}

\vskip 1cm
\section{Fourier coefficients in arithmetic progressions}
\label{Mainestimates}
The aim of this section is to prove Theorem \ref{theo-moment24}. Since we use the same tools as in \cite{Darreye}, 
we will skip some details. For self-contained study of this problem, we refer to the author's Phd thesis \cite{Darreye2}.
 
\vskip 0.5 cm
\subsection{Vorono\u \i\: summation formula}
Let $f(z)=\sum\limits_{n\ge1}a(n)n^{\frac{\ell-1/2}{2}}e(nz)\in S_{\ell+1/2}(4N)$ be a cusp form and let $w$ be a smooth 
$[0,1]$-valued function compactly support in $(0,1)$. Define for any $x>0$, any prime $p\le x$ and any  $a\:[p]$
\begin{equation}
\label{defElastpart}
E(x,p,a)=\frac{1}{\sqrt{x/p}}\sum_{n=a\:[p]}a(n)w(n/x).
\end{equation}

It is shown in \cite{Darreye} that $\frac{1}{\sqrt{x/p}}$ is the right normalization of the sum above since a squareroot cancellation
appears when $x$ and $p$ go to infinity in a certain range.\\

The first step consists in rearranging $E(x,p,a)$ by using the functional equation for $f$ twisted by an additive character. Such
an equation is established in \cite{Hulse} for the special case $N=1$. 
Yet, the proof can be easily adapted to any $N$ and one gets the following.

\begin{prop}
\label{prop-eqfonctwist}
Let $f\in S_{\ell+1/2}(4N)$ as above. Let $u$ and $q$ be two coprime integers such that $(q,4N)=1$. Put
 $$L(s,f,u/q)=\sum_{n\ge1}a(n)e_q(un)n^{-s}$$
then $L(s,f,u/q)$  converges absolutely for {\rm Re} $s>1$ and can be extended  to an entire function satisfying
$$\Lambda(s,f,u/q):= \left(\frac{\sqrt{4N}q}{2\pi}\right)^{s}\Gamma\left(s+\frac{\ell-1/2}{2}\right)L(s,f,u/q)=\omega_q(u)
\Lambda(1-s,f_0,-\overline{4Nu}/q)$$ 
where $u\bar u = 1 \:[q]$ and $\omega_q(u)=\epsilon_q^{-(2\ell+1)}\left(\frac{-\bar u}{q}\right)$. 

Moreover, this $L$-function has polynomial growth in vertical strips. 
\end{prop}

Next, using Mellin transform, we easily deduce the so-called Vorono\u \i\ summation formula.
\begin{prop}
Let $f$, $u$ and $q$ be as above. Then for all $x>0$,
$$\sum_{n\ge1}a(n)e_q(un)w(n/x)=\omega_q(u)\frac{x}{\sqrt{4N} q}\sum_{m\ge1}a_0(m)e_q(-\overline{4Nu}\:m)
B\left(\frac{m}{4Nq^2/x}\right)$$
where $a_0(m)$ is the $m^{\rm th}$  normalized Fourier coefficient of $f_0$ and $B$ is a smooth function of rapid decay as 
in \cite[Section 3]{Darreye}.
\end{prop}

By Mellin transform again, we see that
$$\sum_{n\ge1}a(n)w(n/x)\ll_{f,A} x^{-A}$$ for any $A>0$ so detecting the congruence in the sum in \eqref{defElastpart} and
 applying the last proposition, we have for any $p\nmid 4N$
$$E(x,p,a)=\frac{\epsilon_p^{-(2\ell+1)}}{\sqrt Y}\sum_{m\ge1}a_0(m)\text{Sal}_p(\overline{4N}m,a)B\left(\frac{m}{Y}\right)
+O_{f,A}(x^{-A})$$
where $Y=4Np^2/x$ and
$$\text{Sal}_p(u,v)=\frac{1}{\sqrt p}{\sum_{b\:[p]}}^\times \jacob{b}{p}e_p(ub+v\bar b)$$ is the normalized Salié sum. Classically 
(see \cite[Lemme 8.4.3]{coursKo}), if $u$ and $v$ are coprime to $p$ then 
$$\text{Sal}_p(u,v)=\jacob{v}{p}\epsilon_p\sum_{y^2=uv\:[p]}e_p(2y).$$

Thus, using \cite[Proposition 3]{Darreye} (where in the proof, $f$ does not need to be an eigenform), we infer that
\begin{equation}
\label{Erearrange}
E(x,p,a)=\epsilon_p^{-2\ell}\legendre{a}\frac{1}{\sqrt Y}\sum_{1\le m\le Y^{1+\eta}}a_0(m)\text{Sa}_p(\overline{N}ma)
B\left(\frac{m}{Y}\right)+O_{f,A}(Y^{-A})
\end{equation}
for any $\eta>0$, provided that $Y^{1+\eta}<p$, and where  
$$\text{Sa}_p(y)=\left\{\begin{array}{cc} e_p(\sqrt y^p) + e_p(-\sqrt y^p) & \text{ if } \jacob{y}{p}=1,\\
                                                                              0 & \text{otherwise.}\end{array}\right.$$

\vskip 0.5 cm
\subsection{Some estimates on sums of Fourier coefficients} 

Before proving Theorem \ref{theo-moment24} we establish some basic facts on certain sums of Fourier coefficients.

\begin{lemme}\label{lemme-suma(n)2}
Let $f\in S_{\ell+1/2}(4N)$ as above. Then
$$\sum_{n\ge1}|a(n)|^{2}w(n/x)^{2}\sim c_{f}\|w\|_{2}^{2}x \text{ \;\;\;as $x\to+\infty$}$$
where $$c_{f}= \frac{(4\pi)^{\ell+1/2}}{\Gamma(\ell+1/2){\rm Vol}\left(\Gamma_{0}(4N)\backslash \calH\right)}
\int_{\Gamma_{0}(4N)\backslash \calH}|f(z)|^{2}y^{\ell+1/2}\frac{dxdy}{y^{2}}.$$
\end{lemme}

\begin{proof}
We have for any $\sigma>1$,
$$\sum_{n\ge1}|a(n)|^{2}w(n/x)^{2}=\frac{1}{2i\pi}\int_{(\sigma)}D(s,f\times \bar f)\widehat{w^{2}}(s)x^{s}ds$$
 where $D(s,f\times \bar f)=\sum\limits_{n\ge1}|a(n)|^{2}n^{-s}$ and 
 $\widehat{w^{2}}(s)=\int_{0}^{+\infty}w^{2}(t)t^{s-1}dt$ is the Mellin transform of $w^{2}$. Because $w$ is smooth and 
 compactly supported in $(0,1)$, $\widehat{w^{2}}(s)$ (as well as $\hat w(s)$) is well-defined on the whole complex plane and 
 it is of rapid decay in vertical strips.
 
 Classically (see \cite[Section 13.4]{IwaTopics}), $D(s,f\times \bar f)$ extends to a meromorphic function for $\re\: s\ge1/2$
  with a finite number of poles which are simple and in the interval $1/2<s\le1$. 
  At $s=1$, there is a simple pole whose residue is  $c_{f}$. 
 
Hence, moving the contour of integration to $\sigma=1-\epsilon$ with $ \epsilon>0$ small enough, we get
$$\sum_{n\ge1}|a(n)|^{2}w(n/x)^{2}=c_{f}\widehat{w^{2}}(1)x +
\frac{1}{2i\pi}\int_{(1-\epsilon)}D(s,f\times \bar f)\widehat{w^{2}}(s)x^{s}ds$$
 and because $D(s,f\times \bar f)$ is of polynomial growth on vertical strips, we have the desired conclusion.

\end{proof}

\begin{lemme}\label{lemme-suma(n)}
Let $f\in S_{\ell+1/2}(4N)$ as above and take a class $a\:[p]$. Then
$$\sum_{n=a\:[p]}a(n)w(n/x) = O_{f, w, \epsilon}(x^{- \epsilon}p^{1+ 2 \epsilon})$$
for any $ \epsilon>0$.
\end{lemme}

\begin{proof}
Write
\begin{align*}
\sum_{n=a\:[p]}a(n)w(n/x) &= \frac{1}{p}\sum_{b\:[p]}e_{p}(-ba)\sum_{n\ge1}a(n)e_{p}(bn)w(n/x)\\
&= \frac{1}{p}\sum_{b\:[p]}e_{p}(-ba)\frac{1}{2i\pi}\int_{(\sigma)}L(s,f,b/p)\hat w(s)x^{s}ds
\end{align*}
for any $\sigma>1$. Moving the contour of the integral to $\sigma =- \epsilon$, we get 
$$\sum_{n=a\:[p]}a(n)w(n/x)=\frac{1}{p}\sum_{b\:[p]}e_{p}(-ba)\frac{1}{2i\pi}\int_{(-\epsilon)}L(s,f,b/p)\hat w(s)x^{s}ds$$
 since the integrated functions are entire. Using the functional equation given in Proposition  \ref{prop-eqfonctwist} 
(for $b\neq0\:[p]$), the result follows.

\end{proof}

We will also need to compare sums of Fourier coefficients to Dirichlet series over arithmetic progressions.

\begin{lemme}\label{lemme-sumnalpha}
Let $0<\alpha<1/2$ and take a class $a\:[p]$. Then
$$\sum_{n=a\:[p]}n^{-\alpha}w(n/x)=O\left(\frac{x^{1-\alpha}}{p}\right)$$
\end{lemme}

\begin{proof}
Since $w$ is supported in $(0,1)$ and takes values in $[0,1]$, we have
\begin{align*}
\left|\sum_{n=a\:[p]}n^{-\alpha}w(n/x)\right| &\le \sum_{n\le x/p} \frac{1}{(a+np)^{\alpha}}\\
&\ll \int_{1}^{x/p} \frac{dt}{(a+tp)^{\alpha}}\\
&\ll \frac{x^{1-\alpha}}{p}.
\end{align*}

\end{proof}

We are now ready to prove Theorem \ref{theo-moment24}.

\vskip 0.5 cm
\subsection{Proof of Theorem \ref{theo-moment24}}

Since the computation of the second moment is the same whether the $a(n)$'s (or equivalently the $a_0(n)$'s) are real or not,
we will assume from now that they are. Let $\nu\in\{2,4\}$ and write for $p\nmid 4N$

$$\frac{1}{p}{\sum_{a\:[p]}}^\times E(x,p,a)^\nu= \frac{1}{2}M_\nu^+ + \frac{1}{2}M_\nu^-$$ with
$$M_\nu^\pm = \frac{2}{p}\sum_{\legendre{a}=1} E(x,p,N\mu_\pm a)^\nu 
= \frac{1}{p}{\sum_{b\:[p]}}^{\times} E(x,p,N\mu_\pm b^{2})^\nu$$
and $\mu_\pm$ is any positive integer such that $\legendre{\mu_\pm}=\pm1$.

Then, by \eqref{Erearrange} and \cite[Lemma 5]{Darreye}, we have
\begin{align*}
M_\nu^\pm&=\frac{1}{Y^{\nu/2}}\sum_{\tiny \begin{array}{c}1\le m_i\le Y^{1+\eta}\\
1\le i\le\nu, \legendre{m_i}=\pm1\end{array}}\prod_{i=1}^\nu a_0(m_i)B\left(\frac{m_i}{Y}\right)
\sum_{\mathbf e\in\{\pm1\}^\nu}\delta_p\left(\sum_{i=1}^\nu e_i\sqrt{\mu_\pm m_i}^p\right)  \\
   &\phantom{eeeeeeeeeeeeeeeeeeeeeeeeee}+O_{f,A}\left(\frac{Y^{\nu/2}}{p}+Y^{-A}\right)
\end{align*}
with $Y=4Np^{2}/x$.
If $\nu=2$ then notice that
$$e_1\sqrt{\mu_\pm m_1}^p+e_2 \sqrt{\mu_\pm m_2}^p=0\:[p] \;\Longleftrightarrow\;
\left\{\begin{array}{c}m_1=m_2\\e_1e_2=-1 \end{array}\right.$$
since $1\le\sqrt{\mu_\pm m_i}^p<p/2$ and $1\le m_i <p$. Therefore,

$$M_2^\pm=\frac{2}{Y}\sum_{\tiny \begin{array}{c}1\le m\le Y^{1+\eta}\\ \legendre{m}=\pm1\end{array}}
a_0(m)^2B^2\left(\frac{m}{Y}\right)+O_{f,A}\left(\frac{Y}{p}+Y^{-A}\right)$$ 
and 
$$\frac{1}{p}{\sum_{a\:[p]}}^\times E(x,p,a)^2=\frac{1}{Y}\sum_{1\le m\le Y^{1+\eta}}
 a_0(m)^2B^2\left(\frac{m}{Y}\right)+O_{f,A}\left(\frac{Y}{p}+Y^{-A}\right)$$ 
so, if $Y\to+\infty$ with $Y^{1+\eta}<p$, we get the first assertion of
Theorem \ref{theo-moment24} since, by Lemma \ref{lemme-suma(n)2} or simply \cite[Section 6]{Darreye}, we have 
\begin{equation}
\label{suma(n)2B}
\frac{1}{Y}\sum_{1\le m\le Y^{1+\eta}}a_0(m)^2B^2\left(\frac{m}{Y}\right) \sim c_{f}\|w\|_{2}^{2}\text{ \;\;\;\;as $Y\to+\infty$} 
\end{equation}
and the constant $c_{f}$ is the same as in Lemma \ref{lemme-suma(n)2} because a change of variable shows that $f$ and 
$f_{0}$ have the same Petersson norm.\\

Things are a bit trickier when $\nu=4$. Put
$$Q_4(\mathbf x)=\prod_{\tiny \begin{array}{c}\mathbf e\in\{\pm1\}^4\\ e_1=1\end{array}}\sum_{i=1}^4e_i\sqrt x_i$$
for any positive $x_1$, $x_2$, $x_3$, $x_4$. Because of the parity in the variables $\sqrt x_i$'s of the right-hand side of the above
 equality, we may view $Q_4(\mathbf x)$ as a homogeneous polynomial of $\setZ[x_1,x_2,x_3,x_4]$. 
 
 For $1\le m_1,m_2,m_3,m_4\le Y^{1+\eta}$ we have 

 $$|Q_4( m_1,m_2,m_3,m_4)|\le\prod_{\tiny \begin{array}{c}\mathbf e\in\{\pm1\}^4\\ e_1=1\end{array}}
      4 Y^\frac{1+\eta}{2}\le 2^{16}Y^{4(1+\eta)}.$$

Thus, if $p\ll x^{4/7-\epsilon}$ then there exists $\epsilon'>0$ such that $p^{7/4+2\epsilon'}\ll x $ which implies that
$p^{2(1+\epsilon')}/x\ll p^{1/4}$ so, taking $0<\eta<\epsilon'$, we have

$$2^{16}Y^{4(1+\eta)}<p/2$$ for $Y$ large enough. Assume this is the case, then for any $\mathbf e\in\{\pm1\}^4$ and
any $\mathbf m=(m_1,m_2,m_3,m_4)$ with $1\le m_i\le Y^{1+\eta}$ and $\legendre{m_i}=\pm1$, one has
$$\sum_{i=1}^4 e_i\sqrt{\mu_\pm m_i}^p=0\:[p]\;\Rightarrow\; Q_4(\mu_{\pm}\mathbf m)=0\:[p]
\;\Rightarrow\;  Q_4(\mathbf m)=0\;\Rightarrow\; \exists \mathbf e'\in\{\pm1\}^4, 
\sum_{i=1}^4 e_i'\sqrt m_i=0.$$

For any $1\le i\le4$ and $1\le m_i\le Y^{1+\eta}$, write $m_i=t_ir_i^2$ where $t_i$ is 
squarefree and $r_i\ge1$. Since the different values of the $\sqrt t_i$'s are linearly independent over $\setQ$, then 
 $\sum\limits_{i=1}^4 e_i'\sqrt m_i=0$ only if $|\{t_1,t_2,t_3,t_4\}|=1$ or $2$. In the second case, say $t_1=t_2\neq t_3=t_4$, 
 $$\sum_{i=1}^4 e_i'\sqrt m_i=0\;\Rightarrow\; (e'_{1}r_{1}+e'_{2}r_{2})\sqrt t_{1} + (e'_{3}r_{3}+e'_{4}r_{4})\sqrt t_{3} =0
 \;\Rightarrow\; \left\{\begin{array}{c}r_1=r_2\\e'_1e'_2=-1\\r_3=r_4\\
 e'_3e'_4=-1 \end{array}\right.\;\Rightarrow\;\left\{\begin{array}{c} m_1=m_2\\m_3=m_4 \end{array}\right.$$
since $r_i\ge1$ for all $i$. But if  $m_1=m_2\neq m_3=m_4 $ then 
\begin{align*}
(e_1+e_2)\sqrt{\mu_\pm m_1}^p+(e_3+e_4)\sqrt{\mu_\pm m_3}^p=0\:[p] &\;\Longleftrightarrow\;
 e_1=-e_2 \text{ and } e_3=-e_4.
\end{align*}
Indeed, if, for example, $e_{1}+e_{2}\neq0$ \emph{i.e.} $e_{1}+e_{2}\in\{\pm 2\}$, then $e_{3}+e_{4}\neq0$ 
(otherwise $\sqrt{\mu_\pm m_1}^p=0\:[p]$) and $\mu_{\pm}m_{1}=\mu_\pm m_3 \:[p]$ which implies $m_{1}= m_3$ but
we have excluded this case. This proves the necessary condition of the above equivalence and the sufficient condition is trivial. 
Therefore, 
$$\sum_{\mathbf e\in\{\pm1\}^4}\delta_p\left(\sum\limits_{i=1}^4 e_i\sqrt{\mu_\pm m_i}^p\right)=4.$$

Since this discussion is the same if $m_1=m_3\neq m_2=m_4$ or $m_1=m_4\neq m_2=m_3$, it allows us to write 
\begin{equation}
\label{defM4pm}
M_4^\pm=\frac{12}{Y^{2}}\sum_{\tiny \begin{array}{c}1\le m_1, m_2\le Y^{1+\eta}\\
t_1\neq t_2, \legendre{m_i}=\pm1\end{array}} a_0(m_1)^2B^2\left(\frac{m_1}{Y}\right)a_0(m_2)^2B^2\left(\frac{m_2}{Y}\right)
+\frac{R}{Y^2} +O_{f,A}\left(\frac{Y^{2}}{p}+Y^{-A}\right)
\end{equation}
where
\begin{equation}
\label{defRmoment4}
R=\underset{\tiny \begin{array}{c}t\le Y^{1+\eta}\\ \legendre{t}=\pm1\end{array}}{{\sum}^\flat}
 \sum_{\tiny \begin{array}{c}1\le tr_i^2\le Y^{1+\eta}\\
1\le i\le4\end{array}}\prod_{i=1}^4 a_0(tr_i^2)B\left(\frac{tr_i^2}{Y}\right)
\sum_{\mathbf e\in\{\pm1\}^4}\delta_p\left(\sum_{i=1}^4 e_i\sqrt{\mu_\pm tr_i^2}^p\right).
\end{equation}

We are going to show that this last term is negligible. Precisely we have the following.

\begin{prop}\label{prop-boundR}
Define $R$ as in \eqref{defRmoment4}. Then for any $\epsilon>0$ and $\eta$ sufficiently small
$$|R|\ll_{f,\epsilon}Y^{2-1/7+\epsilon}.$$ 
\end{prop}

\begin{proof}
Following the previous discussion or simply by \cite[Lemma 6]{Darreye},
$$\sum_{\mathbf e\in\{\pm1\}^4}\delta_p\left(\sum_{i=1}^4 e_i\sqrt{\mu_\pm tr_i^2}^p\right)\ll 
\sum_{\mathbf e\in\{\pm1\}^4}\delta_0\left(\sum_{i=1}^4 e_i r_i\right)$$
so, since $B$ is bounded, it suffices to prove that
\begin{equation}
\label{defR'proof}
R':=\underset{t\le Y^{1+\eta}}{{\sum}^\flat} \sum_{\tiny \begin{array}{c}1\le tr_i^2\le Y^{1+\eta}\\
\sum_{i=1}^4 e_i r_i=0\end{array}}\prod_{i=1}^4 |a_0(tr_i^2)|\ll_{f,\epsilon}Y^{2-1/7+\epsilon}
\end{equation}
for any $\mathbf e\in\{\pm1\}^4$ and any $\epsilon>0$. Fix such $\mathbf e$ and $\epsilon$. For any $t$, the inner sum in $R'$
becomes 
$$\sum_{\tiny \begin{array}{c}1\le tr_i^2\le Y^{1+\eta}\\1\le i\le3\end{array}}
\left|a_0\left(t\left(\sum_{i=1}^3e_ir_i\right)^2\right)\right|\prod_{i=1}^3|a_0(tr_i^2)|\ll_{f,\epsilon}t^{6/7+\epsilon}\:
Y^{\epsilon(1+\eta)} \sum_{\tiny \begin{array}{c}1\le tr_i^2\le Y^{1+\eta}\\1\le i\le3\end{array}}
\prod_{i=1}^3r_i^\epsilon$$
by Proposition \ref{prop-boundcoeffgene}. Thus, for $\eta$  sufficiently small,
\begin{align*}
R'&\ll_{f,\epsilon}Y^{\epsilon(1+\eta)}\underset{t\le Y^{1+\eta}}{{\sum}^\flat}t^{6/7+\epsilon}
\left(\frac{Y^{1+\eta}}{t}\right)^{\frac{3}{2}(1+\epsilon)}\\
&\ll_{f,\epsilon}Y^{3/2+6\epsilon}\underset{t\le Y^{1+\eta}}{{\sum}^\flat}t^{5/14-1}\\
&\ll_{f,\epsilon}Y^{2-1/7+7\epsilon}.
\end{align*}

\end{proof}

To finish the proof of Theorem \ref{theo-moment24}, note that
$$\sum_{\tiny \begin{array}{c}1\le m_1, m_2\le Y^{1+\eta}\\
t_1= t_2, \legendre{m_i}=\pm1\end{array}} a_0(m_1)^2B^2\left(\frac{m_1}{Y}\right)a_0(m_2)^2B^2\left(\frac{m_2}{Y}\right)$$
is bounded by $R'$ defined in \eqref{defR'proof}  (with $\mathbf e=(1,-1,1,-1)$ for example) so we have from \eqref{defM4pm} 
and Proposition \ref{prop-boundR},

\begin{align*}
M_4^\pm&=\frac{12}{Y^{2}}\sum_{\tiny \begin{array}{c}1\le m_1, m_2\le Y^{1+\eta}\\
 \legendre{m_i}=\pm1\end{array}} a_0(m_1)^2B^2\left(\frac{m_1}{Y}\right)a_0(m_2)^2B^2\left(\frac{m_2}{Y}\right)
  +O_{f,\epsilon}\left(Y^{-1/7+\epsilon} +\frac{Y^{2}}{p}\right)\\
  &= 12\left(\frac{1}{Y}\sum_{\tiny \begin{array}{c}1\le m\le Y^{1+\eta}\\
 \legendre{m}=\pm1\end{array}} a_0(m)^2B^2\left(\frac{m}{Y}\right)\right)^{2} + 
 O_{f,\epsilon}\left(Y^{-1/7+\epsilon} +\frac{Y^{2}}{p}\right)
\end{align*}
for any $ \epsilon>0$. Recall that $x^{1/2+ \epsilon}\ll p\ll x^{4/7- \epsilon}$ for some $ \epsilon>0$ so the error term above is 
$o(1)$ as $x\to+\infty$. Hence 
$$|M_4^\pm|\le 12 \left(\frac{1}{Y}\sum_{1\le m\le Y^{1+\eta}} a_0(m)^2B^2\left(\frac{m}{Y}\right)\right)^{2} + o(1)$$
 and again, using \eqref{suma(n)2B}, we get the conclusion.

\vskip 1cm
\section{Proof of Theorem \ref{theo-lowboundT}}
\label{ProofofTheo}

We are now going to prove an analog of Theorem \ref{theo-lowboundT} for $\T_{a,q}^+(x,\alpha; w)$ defined in 
\eqref{defTsmooth}. This result will be even stronger than Theorem \ref{theo-lowboundT} since it counts the number
of positive coefficients $a(n)$ with $n/x$ in the support of $w$.

\vskip 0.5 cm
\subsection{Preliminary lemmas}

We first prove two elementary lemmas that we will use several times.

\begin{lemme}\label{lemma-elmt1}
Let $(b(n))_{n}$ be a sequence of real numbers such that $$ \sum_{n\le x}b(n)=o\left(\sum_{n\le x}|b(n)|\right)$$ 
as $x\to+\infty$. Put $${\sum}^{+}(x)= \!\!\sum_{\tiny\begin{array}{c}n\le x\\ b(n)>0\end{array}}b(n) \text{ \;\;and\;\; }
{\sum}^{-}(x)= -\!\!\sum_{\tiny\begin{array}{c}n\le x \\ b(n)<0\end{array}}b(n).$$

Then, $${\sum}^{\pm}(x)\sim \frac{1}{2}\sum_{n\le x}|b(n)|$$ as $x\to+\infty$. 

\end{lemme}

\begin{proof}
We have
\begin{align*}
{\sum}^{+}(x)&= \frac{1}{2}\left({\sum}^{+}(x) + {\sum}^{-}(x)\right) 
+ \frac{1}{2}\left({\sum}^{+}(x)-{\sum}^{-}(x)\right)\\
&= \frac{1}{2}\sum_{n\le x}|b(n)| + \frac{1}{2}\sum_{n\le x}b(n)\\
&\sim \frac{1}{2}\sum_{n\le x}|b(n)|
\end{align*}
by assumptions. The proof is the same for ${\sum}^{-}(x)$.

\end{proof}

\begin{lemme}\label{lemma-elmt2}
Let $X$ be a finite set of positive integers and for any $n\in X$, let $b(n)$ and $c(n)$ be two real numbers with $c(n)\ge0$.
Assume there exists $M>0$ and $V>0$ such that $$ \sum_{n\in X} c(n)\le M \le \sum_{n\in X} b(n) $$ and
$$ \sum_{n\in X} b(n)^2\le V.$$

Then, 
$$\left|\left\{n\in X \;\vline\; b(n)>c(n)\right\}\right| \ge \left(M-\sum_{n\in X} c(n)\right)^{2}V^{-1}$$
\end{lemme}

\begin{proof}
One has 

\begin{align*}
M&\le \sum_{\tiny\begin{array}{c}
n \in X \\ b(n)\le c(n)
\end{array}}b(n)
+ \sum_{\tiny\begin{array}{c}
n \in X \\ b(n)> c(n)
\end{array}}b(n)\\
&\le \sum_{n\in X} c(n) + \left(\sum_{\tiny\begin{array}{c}
n \in X \\ b(n)> c(n)
\end{array}}1\right)^{1/2}\left(\sum_{n\in X}b(n)^{2}\right)^{1/2}
\end{align*}
using Cauchy-Schwarz inequality in the second sum. 

Since $\sum\limits_{n\in X} c(n)\le M$ and 
$ \sum\limits_{n\in X} b(n)^2\le V$, the result follows easily.

\end{proof}

\vskip 0.5 cm
\subsection{Case where $f$ is arbitrary}
Fix $f$ as in Theorem \ref{theo-lowboundT} (but not necessarily an eigenform). For $x>0$ and a prime number $p$, we 
always assume that $ x^{1/2+\epsilon}\ll p\ll x^{4/7-\epsilon} $ for some fixed $ \epsilon>0$. Hence, if $x$ goes to infinity then so 
does $p$ but restricted in this range. We can first establish the following proposition.

\begin{prop}\label{prop-lowboundE(xpa)}
If $0<m< \frac{\| w\|_{2}\sqrt{c_{f}}}{4\sqrt 3}$ then
$$\big|\big\{a\:[p] \;\big|\; E(x,p,a)>m \big\}\big| \ge \left(\frac{1}{4\sqrt 3} - \frac{m}{\| w\|_{2}\sqrt{c_{f}}}\right)^{2}p 
+o(p)$$
as $x\to+\infty$.
\end{prop}

\begin{proof}
By H\"older's inequality, we have
$$\frac{1}{p} {\sum_{a\:[p]}}^{\times}E(x,p,a)^{2}\le \left(\frac{1}{p} {\sum_{a\:[p]}}^{\times}|E(x,p,a)|\right)^{2/3}
\left(\frac{1}{p} {\sum_{a\:[p]}}^{\times}E(x,p,a)^{4}\right)^{1/3} $$
so using Theorem \ref{theo-moment24}, we get
$$c_{f}\| w\|_{2}^{2}+o(1)\le \left(\frac{1}{p} {\sum_{a\:[p]}}^{\times}|E(x,p,a)|\right)^{2/3}
\left(12 (c_{f}\| w\|_{2}^{2})^{2}+o(1)\right)^{1/3}$$
and then
$${\sum_{a\:[p]}}^{\times}|E(x,p,a)|\ge \frac{\| w\|_{2}\sqrt{c_{f}}}{2\sqrt 3}p + o(p).$$

Also, by Lemma \ref{lemme-suma(n)},
$${\sum_{a\:[p]}}^{\times}E(x,p,a) = \frac{1}{\sqrt{x/p}}\sum_{n\ge 1}a(n)w(n/x)-
\frac{1}{\sqrt{x/p}}\sum_{n=0\:[p]}a(n)w(n/x)=O_{f}(x^{-1/2}p^{3/2+\epsilon})=O_{f}(p^{1-\delta})$$
for some $\delta>0$ because $p\ll x^{4/7-\epsilon}$.

Thus, Lemma \ref{lemma-elmt1} yields
\begin{equation}
\label{sumE(xpa)+}
{\sum_{a\:[p]}}^{+}E(x,p,a)\ge \frac{\| w\|_{2}\sqrt{c_{f}}}{4\sqrt 3}p + o(p)
\end{equation}
where ${\sum\limits_{a\:[p]}}^{+}$ means that we restrict the sum to invertible classes  $a\:[p]$ such that 
$E(x,p,a)> 0$. 

Now, use Lemma \ref{lemma-elmt2} with $X=\{0<a<p \;\vline\; E(x,p,a)> 0\}$ and, for
$a\in X$, with $b(a)=E(x,p,a)$ and $c(a)=m <\frac{\| w\|_{2}\sqrt{c_{f}}}{4\sqrt 3}$. By \eqref{sumE(xpa)+} and 
Theorem \ref{theo-moment24}, we obtain
\begin{align*}
\big|\big\{a\:[p] \;\big|\; E(x,p,a)>m \big\}\big| &\ge  \left(\frac{\| w\|_{2}\sqrt{c_{f}}}{4\sqrt 3}p -mp +o(p)\right)^{2}
\big(c_{f}\| w\|_{2}^{2}p + o(p)\big)^{-1}\\
&\ge \left(\frac{1}{4\sqrt 3} - \frac{m}{\| w\|_{2}\sqrt{c_{f}}}\right)^{2}p +o(p).
\end{align*}
\end{proof}

Proposition \ref{prop-lowboundE(xpa)} allows us to give a lower bound for $ \sum\limits_{n=a\:[p]}a(n)w(n/x)$ 
for a certain number of 
$a\:[p]$. We are now going to upper bound $ \sum\limits_{n=a\:[p]}a(n)^{2}w(n/x)^{2}$ for a large number of $a\:[p]$ in order
to apply Lemma \ref{lemma-elmt2} once again.

\begin{prop}\label{prop-upboundvar}
Let $m>0$. Then
$$\big|\big\{a\:[p] \;\big|\; \sum\limits_{n=a\:[p]}a(n)^{2}w(n/x)^{2}>mx/p \big\}\big| \le 
\Big(\frac{c_{f}\| w\|_{2}^{2}}{m}+o(1)\Big)p .$$
\end{prop}

\begin{proof}
This is a straightforward consequence Markov's inequality and Lemma \ref{lemme-suma(n)2}.

\end{proof}

Now, let us prove the main result of this subsection.

\begin{theo}\label{theo-lowboundsmooth}
Let $f$, $x$ and $p$ be as above. For any $a\:[p]$, define $\T_{a,p}^+(x,\alpha; w)$ as in \eqref{defTsmooth}. 
Let $\alpha\in (3/14,1/4]$ and $r < 1/48$. Then, for $x$ large enough
$$\left|\left\{a\:[p] \;\vline\; \T_{a,p}^+(x,\alpha; w)\ge 1\right\}\right| \ge  rp $$
as long as $x^{1-2\alpha+ \epsilon}\ll p\ll  x^{4/7- \epsilon}$ for some $ \epsilon>0$.
\end{theo}

\begin{proof}

Let $m_{1}>0$ and $m_{2}>0$ such that 
$$\sqrt r < \frac{1}{4\sqrt 3} - \frac{m_{1}}{\| w\|_{2}\sqrt{c_{f}}}$$ and
$$\frac{c_{f}\| w\|_{2}^{2}}{m_{2}}< \left(\frac{1}{4\sqrt 3} - \frac{m_{1}}{\| w\|_{2}\sqrt{c_{f}}}\right)^{2} -r .$$

Apply Propositions \ref{prop-lowboundE(xpa)} and \ref{prop-upboundvar} to see that 
\begin{equation}
\label{conditionproofmaintheo}
\sum_{n=a\:[p]}a(n)w(n/x) \ge m_{1}\sqrt{x/p} \text{ \;\;and\;\; } \sum_{n=a\:[p]}a(n)^{2}w(n/x)^{2} \le m_{2}x/p 
\end{equation}
for a certain number of invertible $a\:[p]$ greater than $rp$  for $p$ large enough \emph{i.e.} $x$ large enough.

Also, by Lemma \ref{lemme-sumnalpha}, 
$$\sum_{n=a\:[p]}n^{-\alpha}w(n/x)\ll \frac{x^{1-\alpha}}{p}.$$ The right-hand side of the above
inequality is $o\left(\sqrt{x/p}\right)$ because $\frac{x^{1-\alpha}}{p}= \frac{x^{1/2-\alpha}}{p^{1/2}}\sqrt{x/p}$ 
and $x^{1-2\alpha+ \epsilon}\ll p$.

Hence, for these invertible $a\:[p]$ satisfying \eqref{conditionproofmaintheo}, using Lemma \ref{lemma-elmt2} with \\
$X=\{n=a\:[p]\;|\; w(n/x)\neq 0\}$, $b(n)=a(n)w(n/x)$ and  $c(n)=n^{-\alpha}w(n/x)$, we have for $x$ large enough,
\begin{align*}
\T_{a,p}^+(x,\alpha; w)&\ge \frac{1}{m_{2}x/p}\left(m_{1}\sqrt{x/p} + o\left(\sqrt{x/p}\right)\right)^{2} > 0
\end{align*}
and since $\T_{a,p}^+(x,\alpha; w)$ is an integer, we get the result.

\end{proof}

Theorem \ref{theo-lowboundsmooth} easily implies the first assertion of Theorem \ref{theo-lowboundT}. Unfortunately, the lower
bound for $\T_{a,p}^+(x,\alpha)$ cannot be improved with our method since it only gives 
$$\T_{a,p}^+(x,\alpha)\ge \frac{m_{1}^{2}}{m_{2}}= \frac{m_{1}^{2}}{c_{f}\|w\|_{2}} \frac{c_{f}\|w\|_{2}}{m_{2}}$$ with
$\frac{m_{1}^{2}}{c_{f}\|w\|_{2}}$ and $\frac{c_{f}\|w\|_{2}}{m_{2}}$ both less than $\frac{1}{4\sqrt 3}$ so the right-hand side of the
 above inequality cannot be greater than one.
 
 We also deduce Corollary \ref{cor-sign} from Theorem \ref{theo-lowboundsmooth}.
 
 \begin{proof}[Proof of Corollary \ref{cor-sign}]
 Let $ \epsilon>0$ and $x>0$. For $x$ large enough, there always exists a prime $p$ in the interval 
 $[x^{4/7-2 \epsilon},x^{4/7- \epsilon}]$ by Bertrand's postulate. Then, applying Theorem \ref{theo-lowboundsmooth} with 
 $\alpha=3/14+2 \epsilon$ and $ \epsilon$ small enough, we see that the number of $n\in[1,x]$ such that $a(n)>n^{-\alpha}$ is 
 greater than $rp\ge rx^{4/7-2 \epsilon}$ for fixed $r<1/48$.

 \end{proof}
 
 We now turn our attention to the second assertion of Theorem \ref{theo-lowboundT}, that we will prove using the same technics 
 as previously.

\vskip 0.5 cm
\subsection{Case where $f$ is a complete eigenform} From now on, assume that $f$ is a complete eigenform and that
$x$ and $p$ still satisfy $x^{1/2+ \epsilon}\ll  p\ll x^{4/7- \epsilon}$ for some $ \epsilon>0$. We start by proving the following 
proposition.

\begin{prop}\label{prop-lowboundvar} For $m>0$ and $ \delta>0$, put 
\begin{equation}
\label{defA(xpm)}
\A(x,p,m, \delta)=\left\{a\:[p] \;\vline\;  \begin{array}{c}\sum\limits_{n=a\:[p]}a(n)^{2}w(n/x)^{2}> mx/p\\ 
\sum\limits_{n=a\:[p]}a(n)^{4}w(n/x)^{4}\le \frac{x^{1+ \delta}}{\sqrt p}\end{array}\right\}.
\end{equation}

Then, for $m$ sufficiently small and $x$ large enough, one has
$$\left|\A(x,p,m, \delta)\right| \gg_{f, \delta} x^{-\delta/2} p^{3/4}.$$
\end{prop}

\begin{proof}
First note that, by Cauchy-Schwarz inequality, 
$$\sum_{n=a\:[p]}a(n)^{2}w(n/x)^{2}\le \sqrt{x/p}\left(\sum_{n=a\:[p]}a(n)^{4}w(n/x)^{4}\right)^{1/2}$$
since $w$ is compactly supported in $(0,1)$. It is also $[0,1]$-valued, so using Proposition \ref{prop-4moment}, one gets
$$\sum_{n=a\:[p]}a(n)^{2}w(n/x)^{2}\ll_{f, \delta_{1}} \frac{x^{1+\delta_{1}}}{\sqrt p}$$ for any $ \delta_{1}>0$ and any $a\:[p]$. 
However, if $a\in\A(x,p,m, \delta)$ then we even have
$$\sum_{n=a\:[p]}a(n)^{2}w(n/x)^{2}\ll_{f, \delta} \frac{x^{1+\delta/2}}{p^{3/4}}.$$

By Markov's inequality and Proposition \ref{prop-4moment}, we also have that
$$\left|\left\{a\:[p] \;\vline\;\sum_{n=a\:[p]}a(n)^{4}w(n/x)^{4}> \frac{x^{1+ \delta}}{\sqrt p}\right\}\right|\le \frac{\sqrt p}{x^{\delta_{2}}}$$
for any $0<\delta_{2}<\delta$.

Therefore, using Lemma \ref{lemme-suma(n)2},
$$x\ll_{f} \sum_{a\not\in\A(x,p,m, \delta)}\sum_{n=a\:[p]}a(n)^{2}w(n/x)^{2}
+ \sum_{a\in\A(x,p,m, \delta)}\sum_{n=a\:[p]}a(n)^{2}w(n/x)^{2}$$
and splitting the first sum according to $\sum\limits_{n=a\:[p]}a(n)^{2}w(n/x)^{2}\le mx/p$ or not, we get
$$x\ll_{f, \delta,\delta_{1}} mx + \frac{x^{1+\delta_{1}}}{\sqrt p} \frac{\sqrt p}{x^{\delta_{2}}}+
\frac{x^{1+\delta/2}}{p^{3/4}}\left|\A(x,p,m, \delta)\right|$$
and the result follows by choosing $\delta_{1}<\delta_{2}$ and $m$ small enough.

\end{proof}

We will prove that for most $a\in\A(x,p,m, \delta)$, the coefficients $a(n)$'s with $n=a\:[p]$ have a certain number of positive and 
negative signs. To do  so, we need to bound the number of $a\:[p]$ such that $\left|\sum\limits_{n=a\:[p]}a(n)w(n/x)\right|$ or
$\sum\limits_{n=a\:[p]}a(n)^{2}w(n/x)^{2}$ is too big.

\begin{prop}\label{prop-boundB(xpm)}
For  $ \delta>0$, put 
\begin{equation}
\label{defB(xpm)}
\B(x,p, \delta)=\left\{a\:[p] \;\vline\;  \left|\sum\limits_{n=a\:[p]}a(n)w(n/x)\right|>\frac{x^{1- \delta}}{p^{5/4}} \text{ or }
\sum\limits_{n=a\:[p]}a(n)^{2}w(n/x)^{2}> \frac{x^{1+ \delta}}{p^{3/4}}\right\}.
\end{equation}

Then, for $m>0$ and $\delta>0$ small enough,
$$\left|\B(x,p, \delta)\right| = o\left(\left|\A(x,p,m, \delta\right)\right|)$$
as long as $x^{1/2+ \epsilon}\ll  p\ll x^{4/7- \epsilon}$ for some $ \epsilon>0$
\end{prop}

\begin{proof}
By Chebychev's inequality and Theorem \ref{theo-moment24}, the number of $a\:[p]$ such that
$$\left|\sum_{n=a\:[p]}a(n)w(n/x)\right|>\frac{x^{1- \delta}}{p^{5/4}}$$ is less than
$$ \frac{(c_{f}\|w\|_{2}^{2}+o(1))x}{x^{2-2\delta}p^{-5/2}}\ll_{f}\frac{p^{5/2}}{x^{1-2\delta}}
= x^{-\delta/2} p^{3/4}\frac{p^{7/4}}{x^{1-5\delta/2}}$$
and $\frac{p^{7/4}}{x^{1-5\delta/2}}=o(1)$ for $\delta$ small enough since $p\ll x^{4/7- \epsilon}$.

Similarly, by Markov's inequality and Lemma \ref{lemme-suma(n)2}, the number of $a\:[p]$ such that
$$\sum_{n=a\:[p]}a(n)^{2}w(n/x)^{2}> \frac{x^{1+ \delta}}{p^{3/4}}$$ is less than
$$ \frac{(c_{f}\|w\|_{2}^{2}+o(1))x}{x^{1+\delta}p^{-3/4}}\ll_{f}\frac{p^{3/4}}{x^{\delta}}$$
which is $o(x^{-\delta/2} p^{3/4})$.

\end{proof}

As previously, when $a\in\A(x,p,m, \delta)$, we use H\"older's inequality to give a lower bound on 
$\sum\limits_{n=a\:[p]}|a(n)|w(n/x)$.

\begin{lemme}\label{lemme-lowboundS*}
Let $a\in\A(x,p,m, \delta)$ defined in \eqref{defA(xpm)}. Then
$$\sum_{n=a\:[p]}|a(n)|w(n/x)\ge m^{3/2}\frac{x^{1-\delta/2}}{p^{5/4}}.$$
\end{lemme}

\begin{proof}
H\"older's inequality yields
$$mx/p< \sum_{n=a\:[p]}a(n)^{2}w(n/x)^{2}\le \left(\sum_{n=a\:[p]}|a(n)|w(n/x)\right)^{2/3}
\left(\sum_{n=a\:[p]}a(n)^{4}w(n/x)^{4}\right)^{1/3}.$$

Hence 
$$\sum_{n=a\:[p]}|a(n)|w(n/x)\ge (mx/p)^{3/2}\left(\frac{x^{1+ \delta}}{\sqrt p}\right)^{-1/2}
\ge m^{3/2}\frac{x^{1-\delta/2}}{p^{5/4}}.$$
\end{proof}

We can now prove the main Theorem of this subsection which implies the second assertion of Theorem \ref{theo-lowboundT}.

\begin{theo}\label{theo-lowboundsmoothnotHecke}
Let $f$, $x$ and $p$ be as above. Assume that $f$ is a complete eigenform. For any $a\:[p]$, define $\T_{a,p}^\pm(x,\alpha; w)$ 
as in \eqref{defTsmooth}. Let $\alpha\in (1/8,1/7]$. Then, for any $\delta>0$ small enough and any $x$ large enough 
$$\left|\left\{a\:[p] \;\vline\; \min \big(\T_{a,p}^+(x,\alpha; w),\T_{a,p}^-(x,\alpha; w)\big)
\gg \frac{x^{1-2\delta}}{p^{7/4}}\right\}\right| \gg_{f,\delta}  \frac{p^{3/4}}{x^{\delta/2}} $$
as long as $x^{1/2+ \epsilon}\ll p\ll  x^{4\alpha- \epsilon}$ for some $ \epsilon>0$.
\end{theo}

\begin{proof}
Let $a\in\A(x,p,m, \delta)\backslash \B(x,p, \delta)$. By Propositions \ref{prop-lowboundvar} and \ref{prop-boundB(xpm)}, such
$a\:[p]$ exists for $m$ and $\delta$ small enough and there are $\gg_{f, \delta} x^{-\delta/2} p^{3/4}$ of them.

Lemma \ref{lemme-lowboundS*} implies that
$$\left|\sum_{n=a\:[p]}a(n)w(n/x)\right|\le\frac{x^{1- \delta}}{p^{5/4}}=o\left(\sum_{n=a\:[p]}|a(n)|w(n/x)\right)$$
so, by Lemma \ref{lemma-elmt1} and for $x$ large enough, 
$$\pm\underset{n=a\:[p]}{{\sum}^{\pm}}a(n)w(n/x)\gg \frac{x^{1-\delta/2}}{p^{5/4}}$$
where $\underset{n=a\:[p]}{{\sum}^{\pm}}$ means that we restrict the sum over $n=a\:[p]$ such that $a(n)>0$ or $a(n)<0$ 
respectively.

Also, by Lemma \ref{lemme-sumnalpha}, 
$$\sum_{n=a\:[p]}n^{-\alpha}w(n/x)\ll \frac{x^{1-\alpha}}{p} = 
\frac{x^{1-\delta/2}}{p^{5/4}} \frac{p^{1/4}}{x^{\alpha-\delta/2}}=o\left(\frac{x^{1-\delta/2}}{p^{5/4}}\right)$$
for $\delta$ small enough because $ p\ll  x^{4\alpha- \epsilon}$. Hence, recalling that $a\not\in \B(x,p, \delta)$,
we can apply Lemma \ref{lemma-elmt2} and obtain
$$\T_{a,p}^\pm(x,\alpha; w)\gg \frac{x^{2-\delta}p^{-5/2}}{x^{1+ \delta}p^{-3/4}}= \frac{x^{1-2\delta}}{p^{7/4}}.$$
\end{proof}

\end{document}